\documentclass{amsart}%
\usepackage{amssymb}
\usepackage{amsfonts}
\usepackage{geometry}
\usepackage{amsmath}
\usepackage{graphicx}%
\providecommand{\U}[1]{\protect\rule{.1in}{.1in}}
\newtheorem{theorem}{Theorem}
\theoremstyle{plain}

\newtheorem{definition}[theorem]{Definition}

\newtheorem{lemma}[theorem]{Lemma}

\newtheorem{remark}[theorem]{Remark}

\numberwithin{equation}{section}
\begin{document}
\title[Tops of dyadic grids]{Tops of dyadic grids and $T1$ theorems}
\author[M. Alexis]{Michel Alexis}
\address{Department of Mathematics \& Statistics, McMaster University, 1280 Main Street
West, Hamilton, Ontario, Canada L8S 4K1}
\email{alexism@mcmaster.ca}
\author[E.T. Sawyer]{Eric T. Sawyer}
\address{Department of Mathematics \& Statistics, McMaster University, 1280 Main Street
West, Hamilton, Ontario, Canada L8S 4K1 }
\email{Sawyer@mcmaster.ca}
\author[I. Uriarte-Tuero]{Ignacio Uriarte-Tuero}
\address{Department of Mathematics, University of Toronto\\
Room 6290, 40 St. George Street, Toronto, Ontario, Canada M5S 2E4\\
(Adjunct appointment)\\
Department of Mathematics\\
619 Red Cedar Rd., room C212\\
Michigan State University\\
East Lansing, MI 48824 USA}
\email{ignacio.uriartetuero@utoronto.ca}
\thanks{E. Sawyer is partially supported by a grant from the National Research Council
of Canada}
\thanks{I. Uriarte-Tuero has been partially supported by grant MTM2015-65792-P
(MINECO, Spain), and is partially supported by a grant from the National
Research Council of Canada}

\begin{abstract}
We extend the notion of a dyadic grid of cubes in $\mathbb{R}^{n}$ to include
infinite dyadic cubes. These `tops' of a dyadic grid form a tiling of
$\mathbb{R}^{n}$ which is subject to the constraints similar to those arising
in tiling Euclidean space by (finite) unit cubes. These tops arise in the
theory of two weight norm inequalities through weighted Haar and Alpert wavelets.

\end{abstract}
\maketitle
\tableofcontents

\section{Introduction}

Let $\mu$ be a positive locally finite Borel measure on $\mathbb{R}^{n}$, and
let $\mathcal{D}$ be a dyadic grid on $\mathbb{R}^{n}$. As pioneered by
Nazarov, Treil and Volberg in \cite{NTV}, \cite{NTV4}, etc., the weighted Haar
decomposition associated with a function $f\in L^{2}\left(  \mu\right)  $ and
an integer $N\in\mathbb{N}$ is given by%
\[
\sum_{\substack{Q\in\mathcal{D}\\\ell\left(  Q\right)  <2^{N}}}\bigtriangleup
_{Q}^{\mu}f+\sum_{\substack{Q\in\mathcal{D}\\\ell\left(  Q\right)  =2^{N}%
}}\mathbb{E}_{Q}^{\mu}f,
\]
and converges to $f$ both in $L^{2}\left(  \mu\right)  $ and $\mu$-almost
everywhere. See below for definitions of the Haar projections $\bigtriangleup
_{Q}^{\mu}$ and $\mathbb{E}_{Q}^{\mu}$. It is this form of Haar decomposition
that is used to prove virtually all $T1$ theorems in the literature. However,
in order to estimate the bilinear form%
\begin{align*}
\left\langle T\left(  f\sigma\right)  ,g\right\rangle _{\omega}  &
=\left\langle T\left(  \sum_{\substack{I\in\mathcal{D}\\\ell\left(  I\right)
<2^{N}}}\bigtriangleup_{I}^{\sigma}f\right)  ,\sum_{\substack{J\in
\mathcal{D}\\\ell\left(  J\right)  <2^{N}}}\bigtriangleup_{J}^{\omega
}g\right\rangle _{\omega}+\left\langle T\left(  \sum_{\substack{I\in
\mathcal{D}\\\ell\left(  I\right)  =2^{N}}}\mathbb{E}_{I}^{\sigma}f\right)
,\sum_{\substack{J\in\mathcal{D}\\\ell\left(  J\right)  =2^{N}}}\mathbb{E}%
_{J}^{\omega}g\right\rangle _{\omega}\\
&  +\left\langle T\left(  \sum_{\substack{I\in\mathcal{D}\\\ell\left(
I\right)  <2^{N}}}\bigtriangleup_{I}^{\sigma}f\right)  ,\sum_{\substack{J\in
\mathcal{D}\\\ell\left(  J\right)  =2^{N}}}\mathbb{E}_{Q}^{\omega
}g\right\rangle _{\omega}+\left\langle T\left(  \sum_{\substack{I\in
\mathcal{D}\\\ell\left(  I\right)  =2^{N}}}\mathbb{E}_{I}^{\sigma}f\right)
,\sum_{\substack{J\in\mathcal{D}\\\ell\left(  J\right)  <2^{N}}}\bigtriangleup
_{J}^{\omega}g\right\rangle _{\omega},
\end{align*}
associated with a truncation of a singular integral operator $T$, one must
begin by estimating the final three inner products on the right hand side
above. Since the cubes $I,J\in\mathcal{D}$ in the tiling with $\ell\left(
I\right)  =\ell\left(  J\right)  =2^{N}$, are infinite in number, one must
exercise special care in summing up estimates over these cubes $I$ and $J$.

In this paper, we push the terms $\sum_{\substack{I\in\mathcal{D}\\\ell\left(
I\right)  =2^{N}}}\mathbb{E}_{I}^{\sigma}f$ and $\sum_{\substack{J\in
\mathcal{D}\\\ell\left(  J\right)  =2^{N}}}\mathbb{E}_{J}^{\omega}g$ to
"infinity", where the corresponding sums reduce to finite sums over `tops'
associated to the dyadic grid $\mathcal{D}$, and where these tops form a
tiling of $\mathbb{R}^{n}$ by at most $2^{n}$ infinite cubes. Since the sums
are now finite, no special care is needed in their estimation.

\subsection{Tops of dyadic grids}

We say that a cube $Q$ in $\mathbb{R}^{n}$ is a $PSA$ cube if its sides are
parallel to the coordinate axes. Any union of an increasing sequence $\left\{
Q_{k}\right\}  _{k=1}^{\infty}$ of $PSA$ cubes $Q_{k}$ is either a $PSA$ cube
in $\mathbb{R}^{n}$ or a set of infinite diameter, that we refer to as an
\emph{infinite} $PSA$ cube. We say that a set $Q$ is a \emph{supercube} if it
is either a $PSA$ cube or an infinite $PSA$ cube.

\begin{definition}
A dyadic supergrid $\mathcal{D}$ in $\mathbb{R}^{n}$ is a collection of
supercubes satisfying

\begin{enumerate}
\item The set of all cubes in $\mathcal{D}$ form a dyadic grid.

\item The supercubes in $\mathcal{D}$ are nested in the sense that given
$Q,Q^{\prime}\in\mathcal{D}$, then either $Q\cap Q^{\prime}=\infty$, $Q\subset
Q^{\prime}$ or $Q^{\prime}\subset Q$.
\end{enumerate}
\end{definition}

We shall now investigate the nature of the supercubes of infinite diameter in
a dyadic supergrid. In particular we will show that they are uniquely
determined by the underlying dyadic grid, and that they form a very special
type of tiling of $\mathbb{R}^{n}$.

In order to motivate this investigation, we now recall the construction of
weighted Alpert wavelets in \cite{RaSaWi}, and correct a small oversight
there. Let $\mu$ be a locally finite positive Borel measure on $\mathbb{R}%
^{n}$, and fix $\kappa\in\mathbb{N}$. For each cube $Q$, denote by
$L_{Q;\kappa}^{2}\left(  \mu\right)  $ the finite dimensional subspace of
$L^{2}\left(  \mu\right)  $ that consists of linear combinations of the
indicators of\ the children $\mathfrak{C}\left(  Q\right)  $ of $Q$ multiplied
by polynomials of degree less than $\kappa$, and such that the linear
combinations have vanishing $\mu$-moments on the cube $Q$ up to order
$\kappa-1$:%
\[
L_{Q;\kappa}^{2}\left(  \mu\right)  \equiv\left\{  f=%
{\displaystyle\sum\limits_{Q^{\prime}\in\mathfrak{C}\left(  Q\right)  }}
\mathbf{1}_{Q^{\prime}}p_{Q^{\prime};\kappa}\left(  x\right)  :\int
_{Q}f\left(  x\right)  x^{\beta}d\mu\left(  x\right)  =0,\ \ \ \text{for
}0\leq\left\vert \beta\right\vert <\kappa\right\}  ,
\]
where $p_{Q^{\prime};\kappa}\left(  x\right)  =\sum_{\beta\in\mathbb{Z}%
_{+}^{n}:\left\vert \beta\right\vert \leq\kappa-1\ }a_{Q^{\prime};\beta
}x^{\beta}$ is a polynomial in $\mathbb{R}^{n}$ of degree less than $\kappa$.
Here $x^{\beta}=x_{1}^{\beta_{1}}x_{2}^{\beta_{2}}...x_{n}^{\beta_{n}}$. Let
$d_{Q;\kappa}\equiv\dim L_{Q;\kappa}^{2}\left(  \mu\right)  $ be the dimension
of the finite dimensional linear space $L_{Q;\kappa}^{2}\left(  \mu\right)  $.

Consider an arbitrary dyadic grid $\mathcal{D}$. For $Q\in\mathcal{D}$, let
$\bigtriangleup_{Q;\kappa}^{\mu}$ denote orthogonal projection onto the finite
dimensional subspace $L_{Q;\kappa}^{2}\left(  \mu\right)  $, and let
$\mathbb{E}_{Q;\kappa}^{\mu}$ denote orthogonal projection onto the finite
dimensional subspace%
\[
\mathcal{P}_{Q;\kappa}^{n}\left(  \mu\right)  \equiv
\mathrm{\operatorname*{Span}}\{\mathbf{1}_{Q}x^{\beta}:0\leq\left\vert
\beta\right\vert <\kappa\}.
\]
To obtain a complete set of orthonormal projections, we use the following
definition and lemma.

\begin{definition}
Define a $\mathcal{D}$\emph{-tower} $\Gamma\subset\mathcal{D}$ to be an
infinite sequence $\left\{  I_{m}\right\}  _{m=1}^{\infty}$ of nested dyadic
cubes $I_{m}\subset I_{m+1}$ with side lengths $\ell\left(  I_{m}\right)
=2^{m}$, and define $\operatorname{top}\Gamma\equiv\bigcup_{m=1}^{\infty}%
I_{m}$, which we refer to as the \emph{top} of the tower $\Gamma$.
\end{definition}

\begin{definition}
We define an equivalence relation $\sim$ on $\mathcal{D}$-towers $\Gamma_{1}$
and $\Gamma_{2}$ by $\Gamma_{1}\sim\Gamma_{2}$ if $\Gamma_{1}\cap\Gamma
_{2}\neq\emptyset$.
\end{definition}

Note that the tops of two towers intersect if and only if the towers are
equivalent, and so we can associate to each equivalence class the unique top
of any representative.

\begin{lemma}
For every dyadic grid $\mathcal{D}$, there are at most $2^{n}$ equivalence
classes of $\mathcal{D}$-towers $\left\{  \Gamma_{1},...\Gamma_{T}\right\}  $,
$1\leq T\leq2^{n}$. Moreover, $\mathbb{R}^{n}$ is the disjoint union of the
tops associated with these equivalence classes.
\end{lemma}

\begin{proof}
Each $\operatorname{top}\Gamma$ is a product of $n$ infinite intervals,
$\operatorname{top}\Gamma=\prod_{k=1}^{n}F_{k}$, where $F_{k}=\left(
-\infty,a_{k}\right)  $, $\left[  a_{k},\infty\right)  $ or $\left(
-\infty,\infty\right)  $. Consider the set of $2^{n}$ sequences $\Omega
_{n}=\left\{  -\infty,\infty\right\}  ^{n}$. We say that $\operatorname{top}%
\Gamma$ is associated with a sequence $\theta=\left(  \theta_{1},...\theta
_{n}\right)  \in\Omega_{n}$ if $\theta_{k}$ is an endpoint of $F_{k}$. Since
any two tops associated with the same $\theta$ have nonempty intersection,
they must coincide. Thus the number of tops is at most the number of
$\theta^{\prime}s$ in $\Omega_{n}$. Finally, every $x\in\mathbb{R}^{n}$ is in
some unit cube and hence in some top. This completes the proof of the lemma.
\end{proof}

\begin{remark}
These tops are precisely the supercubes of infinite diameter in the unique
dyadic supergrid containing the dyadic grid $\mathcal{D}$.
\end{remark}

For the standard dyadic grid, the sets $\operatorname{top}\Gamma_{t}$ are
precisely the quadrants in dimension $2$, the octants in dimension $3$, etc...
Now define%
\begin{align*}
&  \mathcal{F}_{\operatorname{top}\Gamma_{t}}^{\kappa}\left(  \mu\right)
\equiv\left\{  \beta\in\mathbb{Z}_{+}^{n}:\left\vert \beta\right\vert
\leq\kappa-1:x^{\beta}\in L^{2}\left(  \mathbf{1}_{\operatorname{top}%
\Gamma_{t}}\mu\right)  \right\}  \ ,\\
&  \ \ \ \ \ \ \ \ \ \ \ \ \ \ \ \text{and }\mathcal{P}_{\operatorname{top}%
\Gamma_{t};\kappa}^{n}\left(  \mu\right)  \equiv\operatorname*{Span}\left\{
x^{\beta}\right\}  _{\beta\in\mathcal{F}_{\operatorname{top}\Gamma_{t}%
}^{\kappa}}\ .
\end{align*}

The following theorem was proved for the standard dyadic grid in
\cite{RaSaWi}\footnote{For the standard dyadic grid, the decomposition into
tower tops used here corrects an omission in \cite{RaSaWi}.}, which
establishes the existence of Alpert wavelets, for $L^{2}\left(  \mu\right)  $
in all dimensions, having the three important properties of orthogonality,
telescoping and moment vanishing.

\begin{theorem}
[Weighted Alpert Bases]\label{main1}Let $\mu$ be a locally finite positive
Borel measure on $\mathbb{R}^{n}$, fix $\kappa\in\mathbb{N}$, and fix a dyadic
grid $\mathcal{D}$ in $\mathbb{R}^{n}$.

\begin{enumerate}
\item Then $\left\{  \mathbb{E}_{\operatorname{top}\Gamma_{t};\kappa}^{\mu
}\right\}  _{t=1}^{T}\cup\left\{  \bigtriangleup_{Q;\kappa}^{\mu}\right\}
_{Q\in\mathcal{D}}$ is a complete set of orthogonal projections in
$L_{\mathbb{R}^{n}}^{2}\left(  \mu\right)  $ and%
\begin{align}
f  &  =\sum_{t=1}^{T}\mathbb{E}_{\operatorname{top}\Gamma_{t};\kappa}^{\mu
}f+\sum_{Q\in\mathcal{D}}\bigtriangleup_{Q;\kappa}^{\mu}f,\ \ \ \ \ f\in
L_{\mathbb{R}^{n}}^{2}\left(  \mu\right)  ,\label{Alpert expan}\\
&  \left\langle \mathbb{E}_{\operatorname{top}\Gamma_{t};\kappa}^{\mu
}f,\bigtriangleup_{Q;\kappa}^{\mu}f\right\rangle =\left\langle \bigtriangleup
_{P;\kappa}^{\mu}f,\bigtriangleup_{Q;\kappa}^{\mu}f\right\rangle =0\text{ for
}P\neq Q,\nonumber
\end{align}
where convergence in the first line holds both in $L_{\mathbb{R}^{n}}%
^{2}\left(  \mu\right)  $ norm and pointwise $\mu$-almost everywhere.

\item Moreover we have the telescoping identities%
\begin{equation}
\mathbf{1}_{Q}\sum_{I:\ Q\subsetneqq I\subset P}\bigtriangleup_{I;\kappa}%
^{\mu}=\mathbb{E}_{Q;\kappa}^{\mu}-\mathbf{1}_{Q}\mathbb{E}_{P;\kappa}^{\mu
}\ \text{ \ for }P,Q\in\mathcal{D}\text{ with }Q\subsetneqq P,
\label{telescoping}%
\end{equation}

\item and the moment vanishing conditions%
\begin{equation}
\int_{\mathbb{R}^{n}}\bigtriangleup_{Q;\kappa}^{\mu}f\left(  x\right)
\ x^{\beta}d\mu\left(  x\right)  =0,\ \ \ \text{for }Q\in\mathcal{D},\text{
}\beta\in\mathbb{Z}_{+}^{n},\ 0\leq\left\vert \beta\right\vert <\kappa\ .
\label{mom con}%
\end{equation}

\end{enumerate}
\end{theorem}

We can fix\ an orthonormal basis $\left\{  h_{Q;\kappa}^{\mu,a}\right\}
_{a\in\Gamma_{Q,n,\kappa}}$ of $L_{Q;\kappa}^{2}\left(  \mu\right)  $ where
$\Gamma_{Q,n,\kappa}$ is a convenient finite index set. Then
\[
\left\{  h_{Q;\kappa}^{\mu,a}\right\}  _{a\in\Gamma_{Q,n,\kappa}\text{ and
}Q\in\mathcal{D}}%
\]
is an orthonormal basis for $L^{2}\left(  \mu\right)  $, with the
understanding that we add an orthonormal basis of each space $\mathcal{P}%
_{\operatorname{top}\Gamma_{t};\kappa}^{n}\left(  \mu\right)  $ if it is
nontrivial. In particular we have from the theorem above that when
$\mathcal{P}_{\operatorname{top}\Gamma_{t};\kappa}^{n}\left(  \mu\right)
=\left\{  0\right\}  $ for all $1\leq t\leq T$, which is the case for the
doubling measures $\mu$ considered in this paper\footnote{$E_{I}^{\sigma
}\left\vert f\right\vert =\frac{1}{\left\vert I\right\vert _{\sigma}}\int
_{I}\left\vert f\right\vert d\sigma\leq\left\Vert f\right\Vert _{L^{2}\left(
\sigma\right)  }\frac{1}{\sqrt{\left\vert I\right\vert _{\sigma}}}$ tends to
$0$ as $\ell\left(  I\right)  $ goes to infinity since doubling measures
satisfy the reverse doubling condition.}, then
\begin{align*}
\left\Vert f\right\Vert _{L^{2}\left(  \mu\right)  }^{2}  &  =\sum
_{Q\in\mathcal{D}}\left\Vert \bigtriangleup_{Q}^{\mu}f\right\Vert
_{L^{2}\left(  \mu\right)  }^{2}=\sum_{Q\in\mathcal{D}}\left\vert \widehat
{f}\left(  Q\right)  \right\vert ^{2},\\
\text{where }\left\vert \widehat{f}\left(  Q\right)  \right\vert ^{2}  &
\equiv\sum_{a\in\Gamma_{Q,n,\kappa}\text{ }}\left\vert \left\langle
f,h_{Q}^{\mu,a}\right\rangle _{\mu}\right\vert ^{2}.
\end{align*}

\begin{remark}
A dyadic grid $\mathcal{D}$ on the real line is a translate of the standard
grid $\mathcal{D}_{0}$ \emph{if and only if} there are exactly two tops, and
if the tops are $\left(  -\infty,a\right)  $ and $\left[  a,\infty\right)  $,
then $\mathcal{D}=\mathcal{D}_{0}+a$.
\end{remark}

In the case $\kappa=1$, the construction in Theorem \ref{main1} reduces to the
familiar Haar wavelets, where we have the following bound for the Alpert
projections $\mathbb{E}_{I;\kappa}^{\mu}$ (\cite[see (4.7) on page 14]%
{Saw6}):
\begin{equation}
\left\Vert \mathbb{E}_{I;\kappa}^{\mu}f\right\Vert _{L_{I}^{\infty}\left(
\mu\right)  }\lesssim E_{I}^{\mu}\left\vert f\right\vert \leq\sqrt{\frac
{1}{\left\vert I\right\vert _{\mu}}\int_{I}\left\vert f\right\vert ^{2}d\mu
},\ \ \ \ \ \text{for all }f\in L_{\operatorname*{loc}}^{2}\left(  \mu\right)
. \label{analogue}%
\end{equation}
In terms of the Alpert coefficient vectors $\widehat{f}\left(  I\right)
\equiv\left\{  \left\langle f,h_{Q;\kappa}^{\mu,a}\right\rangle \right\}
_{a\in\Gamma_{Q,n,\kappa}}$, we thus have%
\begin{equation}
\left\vert \widehat{f}\left(  I\right)  \right\vert =\left\Vert \bigtriangleup
_{I;\kappa}^{\sigma}f\right\Vert _{L^{2}\left(  \sigma\right)  }\leq\left\Vert
\bigtriangleup_{I;\kappa}^{\sigma}f\right\Vert _{L^{\infty}\left(
\sigma\right)  }\sqrt{\left\vert I\right\vert _{\sigma}}\leq C\left\Vert
\bigtriangleup_{I;\kappa}^{\sigma}f\right\Vert _{L^{2}\left(  \sigma\right)
}=C\left\vert \widehat{f}\left(  I\right)  \right\vert . \label{analogue'}%
\end{equation}

\section{A generalization}

Let $\mu$ be a locally finite positive Borel measure on $\mathbb{R}^{n}$, and
let $\mathcal{P}=\operatorname{Span}\left\{  \varphi_{i}\right\}  _{i=1}^{d}$
denote a $d$-dimensional subspace of locally $L^{2}\left(  \mu\right)  $
complex-valued functions, $\varphi_{i}\in L_{\operatorname{loc}}^{2}\left(
\mu\right)  $, which contains the constant function $\mathbf{1}$ on
$\mathbb{R}^{n}$. For example, $\mathcal{P}$ could be the real finite
dimensional linear space of real polynomials of degree less than $\kappa$ used
in the construction of weighted Alpert wavelets above. More generally,
$\mathcal{P}$ could consist of complex polynomials of $z$ (or instead
$\overline{z}$) of degree less than $\kappa$. In any of these cases, the
finite dimensional space $\mathcal{P}$ is invariant under translations,
dilations and rotations, making them well suited to analysis involving
Taylor's polynomial.

Now define%
\[
L_{Q;\mathcal{P}}^{2}\left(  \mu\right)  \equiv\left\{  f=%
{\displaystyle\sum\limits_{Q^{\prime}\in\mathfrak{C}\left(  Q\right)  }}
\mathbf{1}_{Q^{\prime}}p_{Q^{\prime};\mathcal{P}}\left(  x\right)  :\int
_{Q}f\left(  x\right)  \varphi_{i}\left(  x\right)  d\mu\left(  x\right)
=0,\ \ \ \text{for }1\leq i\leq d\right\}  ,
\]
where $p_{Q^{\prime};\mathcal{P}}\left(  x\right)  =\sum_{i=1\ }%
^{d}a_{Q^{\prime};i}\varphi_{i}\left(  x\right)  $ is a linear combination of
the functions $\varphi_{i}$, $1\leq i\leq d$.

\begin{theorem}
[Weighted Alpert Bases]Let $\mu$ be a locally finite positive Borel measure on
$\mathbb{R}^{n}$, fix $\mathcal{P}$ a $d$-dimensional subspace of
$L_{\operatorname{loc}}^{2}\left(  \mu\right)  $ containing the function
$\mathbf{1}$, and fix a dyadic grid $\mathcal{D}$ in $\mathbb{R}^{n}$.

\begin{enumerate}
\item Then $\left\{  \mathbb{E}_{\operatorname{top}\Gamma_{t};\kappa}^{\mu
}\right\}  _{t=1}^{T}\cup\left\{  \bigtriangleup_{Q;\kappa}^{\mu}\right\}
_{Q\in\mathcal{D}}$ is a complete set of orthogonal projections in
$L_{\mathbb{R}^{n}}^{2}\left(  \mu\right)  $ and%
\begin{align*}
f  &  =\sum_{t=1}^{T}\mathbb{E}_{\operatorname{top}\Gamma_{t};\kappa}^{\mu
}f+\sum_{Q\in\mathcal{D}}\bigtriangleup_{Q;\kappa}^{\mu}f,\ \ \ \ \ f\in
L_{\mathbb{R}^{n}}^{2}\left(  \mu\right)  ,\\
&  \left\langle \mathbb{E}_{\operatorname{top}\Gamma_{t};\mathcal{P}}^{\mu
}f,\bigtriangleup_{Q;\mathcal{P}}^{\mu}f\right\rangle =\left\langle
\bigtriangleup_{P;\mathcal{P}}^{\mu}f,\bigtriangleup_{Q;\mathcal{P}}^{\mu
}f\right\rangle =0\text{ for }P\neq Q,
\end{align*}
where convergence in the first line holds both in $L_{\mathbb{R}^{n}}%
^{2}\left(  \mu\right)  $ norm and pointwise $\mu$-almost everywhere.

\item Moreover we have the telescoping identities%
\[
\mathbf{1}_{Q}\sum_{I:\ Q\subsetneqq I\subset P}\bigtriangleup_{I;\kappa}%
^{\mu}=\mathbb{E}_{Q;\mathcal{P}}^{\mu}-\mathbf{1}_{Q}\mathbb{E}%
_{P;\mathcal{P}}^{\mu}\ \text{ \ for }P,Q\in\mathcal{D}\text{ with
}Q\subsetneqq P,
\]

\item and the moment vanishing conditions%
\[
\int_{\mathbb{R}^{n}}\bigtriangleup_{Q;\mathcal{P}}^{\mu}f\left(  x\right)
\ \varphi\left(  x\right)  d\mu\left(  x\right)  =0,\ \ \ \text{for }%
Q\in\mathcal{D},\text{ }\varphi\in\mathcal{P}\ .
\]

\end{enumerate}
\end{theorem}

In the case when $\mathcal{P}$ consists of complex polynomials of $z$ (or
instead $\overline{z}$) of degree less than $\kappa$, we have analogues of
(\ref{analogue}) and (\ref{analogue'}). More generally, if $N\in\mathbb{N}$
and $\mathcal{P}=\operatorname{Span}\left\{  \varphi_{i}\right\}  _{i=1}^{d}$
where each of the functions $\varphi_{i}$ is uniformly finite type $\kappa$,
then the analogues of (\ref{analogue}) and (\ref{analogue'}) hold. Here we say
$\varphi\in C^{\kappa}\left(  \mathbb{R}^{n}\right)  $ is \emph{uniformly
finite type} $\kappa$ if%
\[
\inf_{Q\text{ a cube in }\mathbb{R}^{n}}\inf_{a\in Q}\sum_{\left\vert
\alpha\right\vert <\kappa}\left\vert \frac{\partial^{\left\vert \alpha
\right\vert }\varphi}{\partial x^{\alpha}}\left(  a\right)  \right\vert
\ell\left(  Q\right)  ^{\left\vert \alpha\right\vert }>0.
\]
The inequalities in (\ref{analogue}) and (\ref{analogue'}) then follow from
Taylor's formula%
\begin{align*}
\varphi\left(  x\right)   &  =\sum_{\ell=0}^{\kappa-1}\frac{\left[  \left(
x-a\right)  \cdot\nabla\right]  ^{\ell}}{\ell!}\varphi\left(  a\right)
+O\left(  \left\vert x-a\right\vert ^{\kappa}\right) \\
&  =\sum_{\left\vert \alpha\right\vert <\kappa}c_{\alpha,\kappa}\frac
{\partial^{\left\vert \alpha\right\vert }\varphi}{\partial x^{\alpha}}\left(
a\right)  \left(  x-a\right)  ^{\alpha}+O\left(  \left\vert x-a\right\vert
^{\kappa}\right)  ,
\end{align*}
together with properties of doubling measures proved in \cite{Saw6}. We leave
the straightforward details to the interested reader.

\end{document}